\documentclass{article}
\usepackage{gokul}
\usepackage{authblk}
\usepackage{bigints}
\usepackage{lipsum}
\usepackage{subcaption}
\usepackage{float,graphicx,subcaption}
\usepackage{enumitem,hyperref,nameref}
\usepackage{subfiles}
\usepackage[toc,page]{appendix}
\usepackage[utf8]{inputenc}

\title{Complex asymptotics of the M\"obius energy gradient of symmetric helix pairs}
\author{Max Lipton}

\newcommand{\Addresses}{{
  \bigskip

  Max Lipton, \textsc{Department of Mathematics, Cornell University}\par\nopagebreak
  \textit{E-mail address}: \texttt{ml2437@cornell.edu}

}}

\begin{document}

\maketitle
\begin{abstract}
The M\"obius energy is a well-studied knot energy with nice regularity and self-repulsive properties. Stationary curves under the M\"obius energy gradient are of significant theoretical interest as they they can indicate equilibrium states of a curve under its own forces. In this paper, we consider stationary symmetric helix pairs under the M\"obius energy. Through methods of complex asymptotics, we characterize the limiting behavior of the M\"obius gradient as the coiling ratio tends to infinity: the gradient will diverge in opposing directions depending on whether the radius is less than or greater than $\frac{1}{2}$.

We conclude by discussing the implications to the more general M\"obius-Plateau energy, where the energy of a curve, or pair of curves, includes the area of the minimal surface bounded by them. Symmetric helix pairs bound a helicoid between them, and applying our result shows that stationary helicoids grow to radius $\frac{1}{2}$ from below as the coiling tends to infinity.
\end{abstract}

\section{Introduction}

Knot energies systematically assign real numbers to knots and links. The M\"obius energy first introduced by O'Hara \cite{ohara91} was designed to act as an infinite potential barrier between knot types in the sense that the energy is finite for smooth, tame knots, but the defining integral blows up to infinity when there is a self-intersection. The M\"obius energy of a knot defined by a tame $C^2$ curve $\gamma: I \to \R^3$ is defined by
\begin{align}
    E(\gamma) &= \iint_{I \times I } \left(  \frac{1}{|\gamma(u) - \gamma(v)|^2} - \frac{1}{D(\gamma(u),\gamma(v))^2} \right) \left|\dot{\gamma}(u)\right|\left||\dot{\gamma}(v)\right|dudv,
    \label{mobiusknot}
\end{align}
where $D$ refers to the intrinsic distance along the knot. The intrinsic distance term acts as a regularization allowing the integral to converge. A major result of Freedman, He, and Wang \cite{freedman1994mobius} showed that the energy is invariant under M\"obius transformations, including spherical inversions, which is where the name comes from. Much of the difficulty in studying the M\"obius energy of knots comes from hard analysis of the intrinsic distance term. However, in the case of two disjoint $C^2$ curves $\gamma_1,\gamma_2: \R \to \R^3$, the M\"obius linking energy has the simpler form of 
\begin{align}
E(\gamma_1, \gamma_2) &= \iint_{\R \times \R} \frac{ |\dot{\gamma_1}(u)| | \dot{\gamma_2}(v) | dudv}{|\gamma_1(u) - \gamma_2(v)|^2}. \label{mobiusenergy}
\end{align}
For the M\"obius energy of links, we do not need to subtract an intrinsic distance term in the integrand for the integral to converge. At a given point $\gamma_1(u)$, the $L^2$ gradient of the M\"obius energy is given by the vector-valued integral
\begin{align}
    G_{\gamma_1, \gamma_2}(u) = 2 \int_\R \left[ \frac{2 P_{\dot{\gamma_1}(u)^\perp}(\gamma_2(v) - \gamma_1(u))}{|\gamma_2(v) - \gamma_1(u)|^2} - \textbf{N}_{\gamma_1(u)} \right] \frac{|\dot{\gamma_2}(v)|dv}{|\gamma_2(v) - \gamma_1(u)|^2}.
    \label{mobiusvar}
\end{align}
The gradient vector field along $\gamma_2$ is defined similarly. The curves are stationary when we have
\begin{align}
    G_{\gamma_1,\gamma_2} &\equiv 0 \label{mobstationary}
\end{align}
across both curves.

The derivation of M\"obius gradient of a knot is given in \cite{freedman1994mobius}, and He \cite{he02} showed the equation also holds in the case of links. Here, $P_{\dot{\gamma_1}(u)^\perp}$ refers to the projection onto the plane normal to the tangent vector at $\gamma_1(u)$, and $\textbf{N}_{\gamma_1(u)}$ refers to the normal vector in the Frenet frame along $\gamma_1$. The gradient is defined similarly along $\gamma_2$ by switching $u$ and $v$. A rough heuristic for interpreting \eqref{mobiusvar} is that curves with high gradient are those in which many rescaled pairwise difference vectors are close to Frenet binormals. The integral projects a rescaled difference vector onto the curve's normal plane, and subtracts the Frenet normal, leaving only the binormal. The rescaling obeys an inverse quartic law, so difference vectors between nearby points of the curve which approximate the binormal well will contribute greatly to the gradient integral. The binormal of a helix points roughly in the direction of its axis, which indicates helices provide a fruitful set of examples in studying the dynamics of the M\"obius energy. 

A general helix $\gamma: \R \to \R^3$ with frequency $\omega > 0$ and radius $A$ is defined by the curve
\begin{align}
\label{helicoid}
    \begin{bmatrix}
    x(t) \\
    y(t) \\
    z(t)
    \end{bmatrix} &=
    \begin{bmatrix}
    A \cos(\omega t) \\
    A \sin(\omega t) \\
    t
    \end{bmatrix}.
\end{align}
The radius is allowed to be negative, and varying $A$ across the real numbers parametrizes a helicoid. The main result of this paper concerns the asymptotics of the M\"obius graident for helix pairs $\gamma_1$ and $\gamma_2$, comprising two helices with radii $A$ and $B$ and a common $\omega$, as $\omega \to \infty$.

For fixed $\omega$, helices are invariant under ``screw" transformations which rotate the $xy$-plane by angle $\omega t$ whilst translating in the $z$ direction by $t$. As the screw transformations are a subgroup of Euclidean isometries (and also of the larger M\"obius group of $S^3$), the gradient flow of the M\"obius energy starting at a helix preserves symmetries under the screw transformations (cf. p.42 of \cite{freedman1994mobius}). Therefore, $G_{\gamma_1,\gamma_2}$ is tangent to their common helicoid, and to calculate the entire M\"obius gradient vector fields along $\gamma_1$ and $\gamma_2$, it suffices to compute them at two particular points along the boundary curves and then apply screw transformations. 

Adapting the variational equation \eqref{mobiusvar} to the case of helix pairs is a matter of calculating all of the components and piecing them together. The author, in joint work with Gokul Nair, carried out this computation in the more general M\"obius-Plateau energy in \cite{ln22}, which is the original motivation of this paper. We discuss the implications to the M\"obius-Plateau energy after our main result. 

The three equations defining the M\"obius gradient vector at $\gamma_1(0)$ are
\begin{align}
    \int_{-\infty}^\infty \left[ \frac{2 \left( B \cos(\omega v) - A \right)}{A^2 - 2AB\cos(\omega v) + B^2 + v^2} - 1 \right] \frac{1}{A^2 - 2AB\cos(\omega v) + B^2 + v^2} dv &= 0 \label{gamma1var}\\
    \int_{-\infty}^\infty \frac{B \sin(\omega v) - \frac{v + \omega^2 A^2 B \sin(\omega v)}{\omega^2 A^2 + 1}}{(A^2 - 2AB\cos(\omega v) + B^2 + v^2)^2}dv &= 0 \nonumber\\
    \int_{-\infty}^\infty \frac{ v - \frac{\omega A B \sin(\omega v) + v}{\omega^2 A^2 + 1}}{\left( A^2 - 2AB\cos(\omega v) + B^2 + v^2 \right)^2} dv &= 0.\nonumber
\end{align}
We refer the reader interested in teh derivation to \cite{ln22}. The last two integrands are odd functions in $v$, which means the integrals are always going to be zero, so only equation \eqref{gamma1var} relevant to us. Through identical computations by taking $v = 0$, noting that in this case we have $\textbf{N}_{\gamma_2(0)} = (-1,0,0)$ and $\textbf{n} = (1, 0, 0)$, we see the first component of the variational equation at $\gamma_2(0)$ is
\begin{align}
     \int_{-\infty}^\infty \left[ \frac{2 \left( A \cos(\omega u) - B \right)}{A^2 - 2AB\cos(\omega u) + B^2 + u^2} + 1 \right] \frac{1}{A^2 - 2AB\cos(\omega u) + B^2 + u^2} du &= 0. \label{gamma2var}
\end{align}
Relabelling the variable of integration and adding \eqref{gamma1var} and \eqref{gamma2var} yields
\begin{equation}
\label{screwsum}
    \int_{-\infty}^\infty \frac{(A + B)(\cos(\omega v) - 1)}{\left(A^2 - 2AB \cos(\omega v) + B^2 + v^2 \right)^2}dv = 0.
\end{equation}
From direct inspection, it is clear that $A = -B$ is a solution, which we call a symmetric helix pair. There exist nonsymmetric solutions to \eqref{screwsum} which can be found numerically, and determining conditions for a nonsymmetric solution to be stationary, remains open.

\section{M\"obius-Stationary Symmetric Double Helices}

 The fact that symmetric double helixes with $A = -B$ satisfy \eqref{screwsum} suggests this class of curve pairs may contain stationary pairs. However, for a symmetric double helix to be stationary, it must also satisfy the difference between equations \eqref{gamma2var} and \eqref{gamma1var}.

This equation is
\begin{equation}
    \label{screwdiff}
  \int_{-\infty}^\infty \frac{A(A+1) + B(B-1) + (A - B -2AB) \cos(\omega v) + v^2}{\left(A^2 - 2AB \cos(\omega v) + B^2 + v^2 \right)^2}dv = 0.
\end{equation}
Substituting $A = -B$ gives
\begin{equation}
    \label{mobsymscrewdiff}
 \bigintsss_{-\infty}^\infty \frac{4B(B-1)\cos^2\left(\frac{\omega v}{2}\right) + v^2}{\left(\left(2B\cos\left(\frac{\omega v}{2}\right)\right)^2  + v^2 \right)^2}dv = 0.
\end{equation}
For $\omega, B > 0$, let $M(\omega, B)$ denote the value of the definite integral in \eqref{mobsymscrewdiff}. Note that for $v^2 > |4B(B-1)|$, the integrand  of $M(\omega,B)$ is always going to be positive. Therefore, the only way for the integral to be nonpositive is for $B$ and $B-1$ to have differing signs in order there to be negative contributions to the integral for small $|v|$. This gives us our first constraint on stationary symmetric helix pairs.
\begin{lemma}
\label{mobiusbound}
    Any stationary symmetric helix pair must have radius $0 < B < 1$.
\end{lemma}
Our main theorem is a considerable strengthening of the observation of Lemma \ref{mobiusbound}.
\begin{theorem}
For any $B > 0$,
\begin{equation}
    \lim\limits_{\omega \to \infty} M(\omega, B) = \left\{ \begin{array}{ll}
      -\infty & B < \frac{1}{2} \\[1ex]
      0 & B = \frac{1}{2} \\[1ex]
      \infty & B > \frac{1}{2} \\[1ex]
\end{array}.
\right.
    \label{MwB}
\end{equation}
\end{theorem}
The proof of our theorem will make use of complex asymptotics. We will compute a contour integral of a $\omega$-dependent meromorphic function with countably infinite poles. For finite $\omega$, the infinite sum of the contributions of the poles to the integral will converge, as we can see directly from the definition of $M$. However, as $\omega \to \infty$, the terms in the series will tend towards those in a divergent sum (except when $B = \frac{1}{2}$), whose limit is given in \eqref{MwB}.

\begin{proof}
We will take a limit of complex contour integrals around rectangles which will tend towards encompassing the entire upper half plane. In the variable $z$, let $F(z)$ denote the integrand of $M(\omega, B)$, which has partial fraction decomposition
\begin{align}
    F(z) &= -\frac{i(2B - 1)}{4Bz\left(2B\cos\left(\frac{\omega z}{2}\right) - iz \right)} + \frac{i(2B - 1)}{4Bz\left(2B\cos\left(\frac{\omega z}{2}\right) + iz \right)} \nonumber \\
    &\hspace{0.5in}+ \frac{1}{4B\left(2B\cos\left(\frac{\omega z}{2}\right) - iz \right)^2} -
    \frac{1}{4B\left(2B\cos\left(\frac{\omega z}{2}\right) + iz \right)^2}.
    \label{partialfrac}
\end{align}
Label these four terms $F_1(z),F_2(z),F_3(z)$, and $F_4(z)$ respectively. Also define $F_- = F_1 + F_3$ and $F_+ = F_2 + F_4$, where the subscript denotes the sign of $iz$ in the denominator. Note that each of these functions is in reality a family of functions depending on $\omega$ and $B$. Throughout this proof, we will fix $B$, but take freely vary $\omega$ and in some instances we will pass to the limit $\omega \to \infty$. The first two terms have a pole at $z = 0$, but as we can see that the integrand in \eqref{mobsymscrewdiff} is well defined at $0$ for $B \neq 0$, this singularity is removable from the sum.

\begin{figure}[h]
\captionsetup[subfigure]{justification=centering}
    \centering
    \begin{subfigure}{0.49\textwidth}
    \centering
    \includegraphics[width=3.05in]{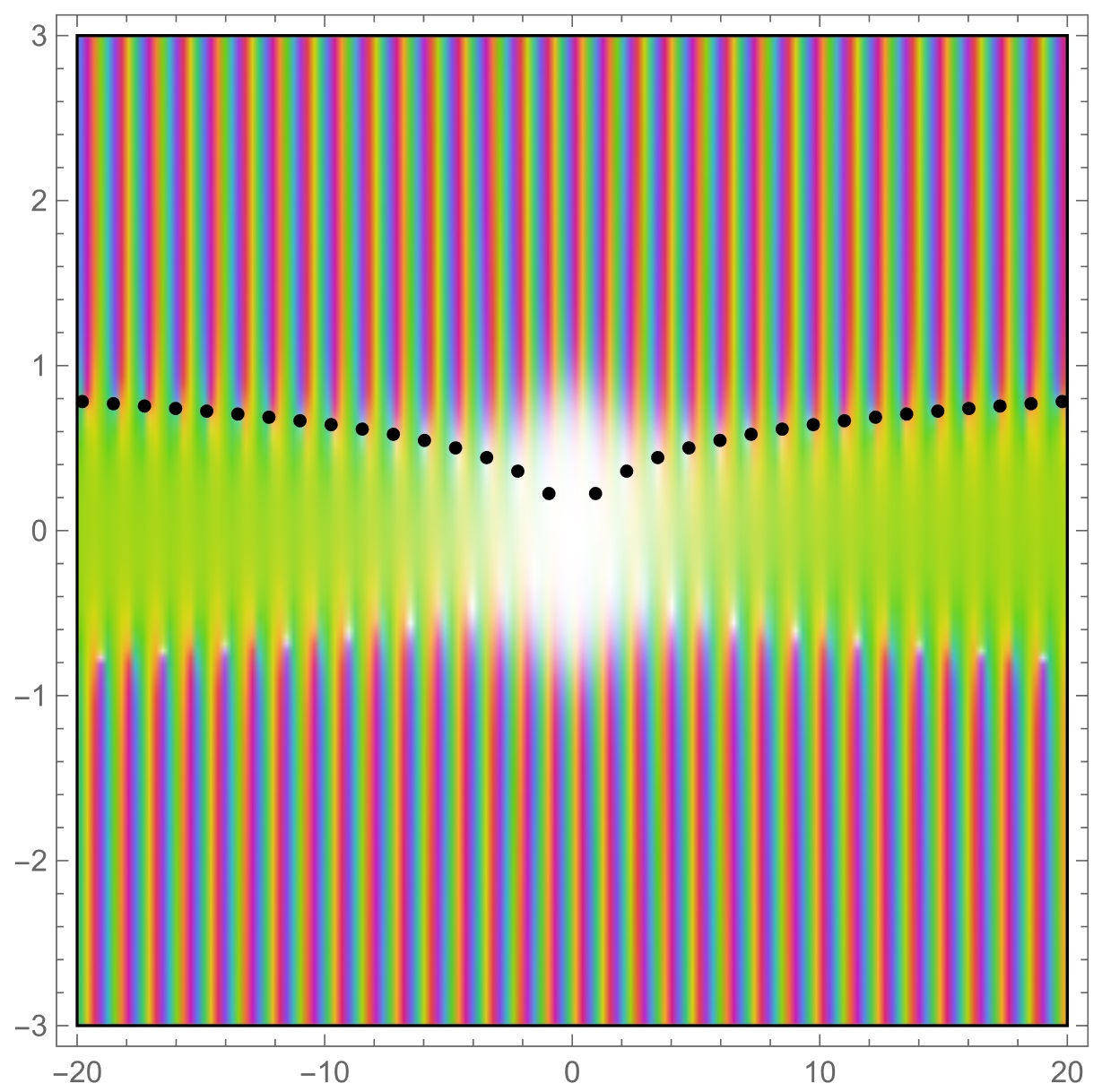}
    \caption{A complex plot of $F_1$ along with locations of the poles in the upper half plane.}
    \end{subfigure}%
    \begin{subfigure}{0.49\textwidth}
    \centering
    \includegraphics[width=3in]{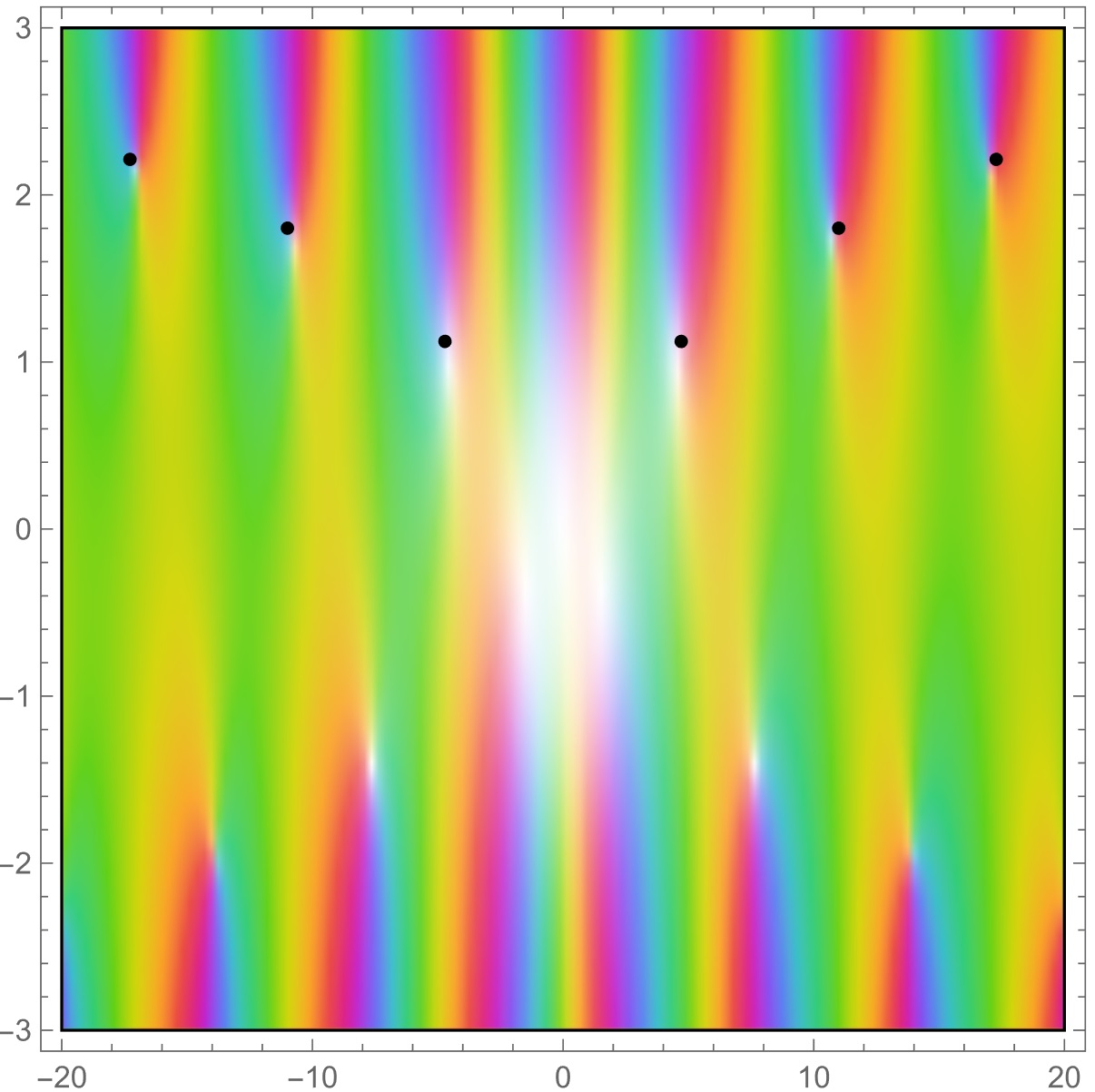}
    \caption{A complex plot of $\widetilde{F_1}$ along with locations of the poles in the upper half plane.}
    \end{subfigure}
    \caption{In both complex plots over the same domain in $\C$, we set $B = 0.4$ and $\omega = 10$. As $\omega \to \infty$, the poles of $F_1$ cluster and tend towards the real axis, whilst the poles in $\widetilde{F_1}$ tend towards every other odd integer multiple of $\frac{\pi}{2}$. Analogous behavior holds for the other $F_i$.}
    \label{F1plots}
\end{figure}

For each term in \eqref{partialfrac}, we will continuously dilate the domain by a factor of $\frac{2}{\omega}$. Let $\widetilde{F_i}(z) = \frac{2}{\omega} F_i\left( \frac{2 z}{\omega} \right)$. Observe that $z_0$ is a pole of $\widetilde{F_i}$ if and only if $\frac{2}{\omega}z_0$ is a pole of $F_i$ and
\begin{align}
\label{resdilated}
\text{Res} \left( \widetilde{F_i}(z), z_0 \right) &= \text{Res} \left( F_i(z), \frac{2}{\omega} z_0 \right).
\end{align}

The rescalings give us
\begin{align}
\widetilde{F_1}(z) &= -\frac{i(2B - 1)}{8B^2z\left( \cos(z) - \frac{iz}{\omega B} \right)}\\
\widetilde{F_2}(z) &= \frac{i(2B - 1)}{8B^2z\left( \cos(z) + \frac{iz}{\omega B} \right)} \\
\widetilde{F_3}(z) &= \frac{1}{8\omega B^3\left( \cos(z) - \frac{iz}{\omega B} \right)^2} \\
\widetilde{F_4}(z) &= -\frac{1}{8\omega B^3\left( \cos(z) + \frac{iz}{\omega B} \right)^2}.
\end{align}

Notice that the pairs $\widetilde{F_1}$ with $\widetilde{F_3}$, and $\widetilde{F_2}$ with $\widetilde{F_4}$ share the same nonzero poles. We claim that each of these functions have precisely one pole in the upper half plane in every other strip $(n\pi,(n+1)\pi) \times (0,\infty)$, provided $\omega$ and $B$ satisfy a weak requirement which will be fulfilled when $\omega \gg B$.

\begin{lemma}
There exists a constant $C < 1$ such that if $\omega B > C$, then $\widetilde{F_1}$ and $\widetilde{F_3}$ have exactly one common nonzero pole in the upper half-plane in each strip bounded by the vertical lines $x = n \pi$ and $x = (n+ \emph{sgn}(n)) \pi$ where $n$ is an odd integer. Additionally, $\widetilde{F_2}$ and $\widetilde{F_4}$ have exactly one common pole in the upper half-plane in each strip bounded by the vertical lines $x = n\pi$ and $x = (n+ \emph{sgn}(n)) \pi$ for each $n$ a nonzero even integer, along with one common pole in each of the strips bounded by $x = 0$ and $x = \pm \pi$.
\label{uniquepoles}
\end{lemma}

\begin{figure}[h]
    \centering
    \includegraphics[width=3in]{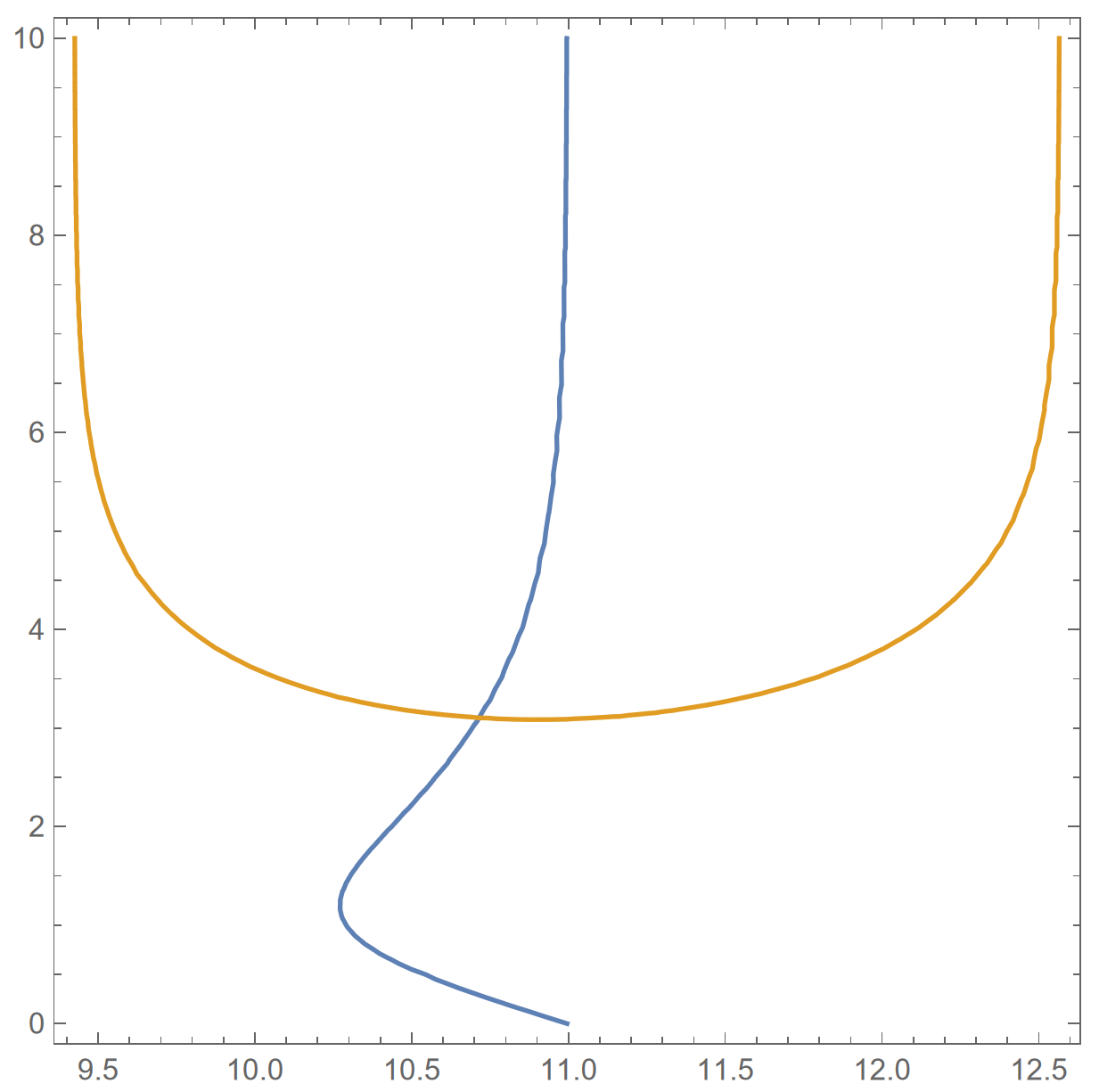}
    \caption{Branches of the curves defined by \eqref{F1y} and \eqref{F1x} (with $\omega = 2, B = \frac{1}{2}$), in blue and orange respectively, in the strip $(3\pi,4\pi) \times (0,\infty)$. The intersection corresponds to a pole of $\widetilde{F_1}$ in the upper half-plane, and there is exactly one pole in every other strip.}
    \label{F1xF1ycurves}
\end{figure}
\begin{proof}
We will only prove this lemma for $\widetilde{F_1}$, as the remaining cases follow analogous reasoning. For $\widetilde{F_1}$, it suffices to identify the zeros of the factor $\cos(z) - \frac{iz}{\omega B}$ in the denominator. Write $z = x + iy$, for $x,y \in \R$. Rewriting the equation $\cos(z) = \frac{iz}{\omega B}$ and equating real and imaginary components gives the system
\begin{align}
    \cos(x)\cosh(y) &= -\frac{y}{\omega B}\label{F1y}\\ 
    \sin(x)\sinh(y) &= -\frac{x}{\omega B}. \label{F1x}
\end{align}
Let $\Gamma_C$ and $\Gamma_S$ denote the solution curves to \eqref{F1y} and \eqref{F1x} respectively (the subscripts refer cosine and sine), restricted to the upper half-plane with $y \geq 0$. We can rewrite \eqref{F1x} as 
\begin{align}
    \Gamma_{S,y}(x) &= \text{arcsinh}\left( \frac{-x}{\omega B \sin(x)} \right),
    \label{xbranch}
\end{align}
expressing $y$ as an even function of $x$. There are vertical asymptotes along each line $x = 2n\pi, n \in \Z - \{0\}$, but between these asymptotes, $y(x)$ is continuous. Furthermore, as $\text{arcsinh}(u) = 0$ if and only if $u = 0$, we can see that between these asymptotes, $y$ does not change signs because $-\frac{x}{\omega B \sin(x)}$ is never zero for nonzero $x$. As $\text{sgn}(\text{arcsinh}(u)) = \text{sgn}(u)$, we have that $\Gamma_S$ is in the upper half plane if and only if $x \in (\pi, 2\pi)\cup (3\pi, 4\pi)\cup \dots$ when $x > 0$, or $x \in (-2\pi,-\pi)\cup (-4\pi, -3\pi)\cup \dots$ when $x < 0$.

Rewriting \eqref{F1y}, we can express $x$ as a function of $y$ with
\begin{align}
\Gamma_{C,x}(y) &= \arccos\left( \frac{-y}{\omega B \cosh(y)} \right).
\label{ybranch}
\end{align}
Since $\omega B$ is not small (taking $\omega B > 0.7$ will work) we have that $\frac{-y}{\omega B \cosh(y)} \in [-1,1]$ for all $y$. Therefore \eqref{ybranch} is well defined. However, the standard definition of $\arccos(y)$ with range $[0,\pi]$ is merely the principal branch of a multivalued function, and to obtain the other branches of $\Gamma_C$, we repeatedly reflect the graph along the lines $x = n\pi$, $n \in \Z$. As $\lim\limits_{y \to \infty} \frac{-y}{\omega B \cosh{y}} = 0$, we can see that $\Gamma_C$ has an asymptote at $x = \arccos(0) =  \frac{\pi}{2}$, and after reflecting the curve over the lines $x = n\pi$, we get additional asymptotes at the lines $x = \frac{n\pi}{2}$. In the principal branch, $\frac{\pi}{2} < x < \pi$ for all $y>0$, so $x$ asymptotically converges to $\frac{\pi}{2}$ from above as $y \to \infty$. Therefore, the branch of $\Gamma_C$ obtained by reflecting the principal branch across $x = \pi$ has the property that  $\pi < x < \frac{3\pi}{2}$ for all $y > 0$, with $x$ asymptotically approaching $\frac{3\pi}{2}$ from below as $y \to \infty$. For $x > 0$, we saw that $\Gamma_S$ resides in the upper half plane  $x \in (\pi, 2\pi)\cup (3\pi, 4\pi)\cup \dots$, and therefore the branches of of $\Gamma_C$ approach their $x$-asymptotes from below. When $x < 0$, the asymptotes are approached from above.

Now fix one of the strip subdomains of the upper half-plane $\Sigma_n = (n\pi, (n+1)\pi) \times (0, \infty)$, with $n > 0$ odd. (The case for $x < 0$ is analogous.) The corresponding branch of $\Gamma_S$, expressing $y$ as a function of $x$, is well-defined on the entire open interval and contained entirely in this $\Sigma_n$. As the branch of $\Gamma_C$ has an $x$-asymptote at $n\pi + \frac{\pi}{2}$ while also containing the point $\left( n\pi + \frac{\pi}{2}, 0 \right)$, the two branches must intersect, indicating the existence of a pole of $\widetilde{F_1}$ in this strip, viewed as a subset of $\C$. We now claim this intersection is unique, as depicted in Figure \ref{F1xF1ycurves}.

This branch of $\Gamma_C$, expressing $x$ as a function of $y$, is defined by
\begin{align}
\Gamma_{C,x}(y) &= (n+1)\pi - \arccos\left(\frac{-y}{\omega B \cosh(y)} \right),
\label{ybranch}
\end{align}
and from this point on we will only use the notation $\Gamma_{C,x}$ to refer to the function of $y$ defining this branch. Differentiating with respect to $y$, we get
\begin{align}
    \frac{d}{dy} \Gamma_{C,x}(y) &= \frac{ (y \tanh(y) - 1)\text{sech}(y)}{\omega B \sqrt{1 - \frac{y^2 \text{sech}^2(y)}{\omega^2 B^2}}}. \label{dxdy}
\end{align}
Let $y_0$ be the unique positive solution to $y\tanh(y) - 1 = 0$, which is approximately $y_0 \approx 1.199 < 1.2$. For $y > y_0$, observe that \eqref{dxdy} is strictly positive and well-defined. By the Inverse Function Theorem, the subset of the branch of $\Gamma_C$ where $y > 1.2$ can also be expressed the graph of a monotonic function of $x$ $\Gamma_{C,y}$ over some subdomain of $\left(n\pi,\frac{(n+1)\pi}{2}\right)$, given by the inverse function of \eqref{ybranch}. We are able to restrict the righthand limit of this subdomain because we saw $\Gamma_C$ does not intersect $\left(\frac{(n+1)\pi}{2}, (n+1)\pi\right) \times (0,\infty)$. By inverting \eqref{dxdy}, we can see that this inverse function expressing $y$ as a function of $x$ has strictly positive derivative with respect to $x$, as depicted by the blue curve on \ref{F1xF1ycurves}. 

To summarize, we have deduced that any intersection of the two branches of \eqref{F1y} and \eqref{F1x} must reside in the substrip $\Sigma'_n = \left(n\pi, n\pi + \frac{\pi}{2}\right) \times \big[ 1.2, \infty)$. Within $\Sigma'_n$, both branches are graphs of functions of $x$. We have already deduced that an intersection exists, but to prove the intersection is unique, we will establish that within the $\Sigma'_n$, the derivative of the function defining $\Gamma_C$ with respect to $x$ (when it is defined) is strictly greater than the derivative of $\Gamma_{S,y}$.

Differentiating \eqref{xbranch} gives
\begin{align}
    \frac{d}{dx} \Gamma_{S,y}(x) &= \frac{\csc^2(x)(x\cos(x) - \sin(x))}{\omega B \sqrt{\frac{x^2\csc^2(x)}{\omega^2B^2} + 1}}.
    \label{dydx}
\end{align}
Recall that $\sin(x) < 0$ for all $x \in (n\pi, (n+1)\pi)$, and $\cos(x) < 0$ within the restricted $x$-domain of $\left( n\pi, n\pi + \frac{\pi}{2} \right)$. Provided $x\cos(x) - \sin(x) < 0$, we will have that \eqref{dydx} is negative and hence less than the strictly positive derivative of $\Gamma_C$ with respect to $x$. Unfortunately, there is a solution to $x\cos(x) - \sin(x)$ between $n\pi$ and $n\pi + \frac{\pi}{2}$, so there is a thin substrip of $\Sigma'_n$ where both derivatives are positive. However, by estimating \eqref{dydx} on the portion where it could be positive, we will see that it is smaller than the $x$-derivative of $\Gamma_C$.

Let $x_n$ denote the solution to $x\cos(x) -\sin(x) = 0$ such that $x \in (n\pi, (n+1)\pi)$. Equivalently, $x_n$ is the solution of $\arctan(x) = x$. Using a first order Taylor expansion of $\arctan(x) - x$ as shown in \cite{frankel37}, we can get
\begin{align}
    x_n &> \frac{(2n+1)\pi}{2} - \frac{2}{(2n+1)\pi} \nonumber \\
    &> n\pi + \frac{\pi}{2} - \frac{1}{n\pi} \label{xtanx}.
\end{align}
Let $I_n$ denote the interval $\left(n\pi + \frac{\pi}{2} - \frac{n}{\pi}, n\pi + \frac{\pi}{2} \right)$. These are intervals of $x$-values where \eqref{dydx} could be positive within the $\Sigma'_n$. The widths tend to zero as $n \to \infty$, with the largest width being $\frac{1}{\pi}$ when $n = 1$. The numerator of \eqref{dydx} is bounded above by $1$ for $x \in I_n$. It is an elementary calculus exercise to show that
\begin{align*}
    \inf\limits_{x \in I_n} x^2\csc^2(x) &> \inf\limits_{x \in I_1} x^2 \csc^2(x) \\
    &\approx 21.19 > 21.
\end{align*}
Therefore, we have 
\begin{align}
    \sup\limits_{x \in I_n} \frac{d}{dx} \Gamma_{S,y}(x) &< \frac{1}{\sqrt{\omega^2 B^2 + 21}}. \label{supbranchx}
\end{align}
To get a lower bound on the $x$-derivative of $\Gamma_{C,y}(x)$ over $I_n$, it suffices to compute
\begin{align*}
    \inf\limits_{y > 1.2} \frac{d}{dy} \left( \Gamma_{C,x}(y) \right)^{-1} &= \inf\limits_{y > 1.2} \frac{\omega B \sqrt{1 - \frac{y^2 \text{sech}^2(y)}{\omega^2 B^2}}}{ (y \tanh(y) - 1)\text{sech}(y)}.
\end{align*}
First, it is easy to calculate that $y^2\text{sech}^2(y) < 0.5$ for all $y$. Therefore,
\begin{align*}
    \inf\limits_{y > 1.2} \omega B \sqrt{1 - \frac{y^2 \text{sech}^2(y)}{\omega^2 B^2}} &> \sqrt{\omega^2B^2 - 0.25}.
\end{align*}
Next, we can calculate a global maximum to get
\begin{align*}
    \sup\limits_{y > 1.2} (y \tanh(y) - 1) \text{sech}(y) &\approx 0.2511 \\
    &< 0.26.
\end{align*}
We now conclude
\begin{align}
    \inf\limits_{x \in I_n} \frac{d}{dx} \Gamma_{C,y}(x) &> \frac{50}{13} \sqrt{\omega^2 B^2 - 0.26}.
    \label{infbranchy}
\end{align}
We can conclude \eqref{infbranchy} exceeds \eqref{supbranchx} provided $\omega B > 0.7$. The constants obtained in the other cases are of comparably small size.
\end{proof}

Consider the discrete sets
\begin{align}
    \Lambda^-(\omega) &= \left\{ 2n\pi - \frac{\text{sgn}(n) \pi}{2} + i\text{arcsinh}\left(\frac{2 \left| n \right|\pi}{\omega B} - \frac{\pi}{2\omega B} \right) \colon n \in \Z - \{0\} \right\} \\
    \Lambda^+(\omega) &= \left\{ 2n\pi + \frac{\text{sgn}(n) \pi}{2} + i\text{arcsinh}\left(\frac{2 \left| n \right| \pi }{\omega B} + \frac{\pi}{2\omega B} \right) \colon n \in \Z - \{0\}\right\} \cup \left\{ \pm \frac{\pi}{2} \right\}
\end{align}
Denote the elements of $\Lambda^{\pm}$ by $\lambda^{\pm}_n$, for $n \in \Z - \{0\}$. Through an abuse of notation, we can let $\lambda^-_{\pm 0}$ denote $\pm \frac{\pi}{2}$, for positive and negative zero. These are approximate locations for the upper-half plane poles $\widetilde{F_-}$ and $\widetilde{F_+}$ respectively. These are not the exact poles, but we will now show that these points are good approximations for the poles in that they are asymptotically equivalent to the exact poles as $\omega \to \infty$.

\begin{lemma} Let $z^{\pm}_n$ denote the nonzero upper half-plane poles of $\widetilde{F_i}$ as described in Lemma \ref{uniquepoles}. Then for all $n$,
\begin{align*}
    \lim\limits_{\omega \to \infty} \frac{\lambda^{\pm}_n}{z^{\pm}_n} &= 1.
\end{align*}
\end{lemma}
\begin{proof}

Fix $n > 0$, and we will work with $\Lambda^-$, dropping the superscripts. The other cases follow analogous reasoning, which we will omit. In this case, $\lambda_n = \frac{(4n-1)\pi}{2} + i \text{arcsinh}\left( \frac{(4n-1)\pi}{2\omega B} \right)$. We also have that $z_n$ is a zero of the function $\cos(z) - \frac{iz}{\omega B}$. First, we claim that
\begin{align}
    \lim\limits_{\omega \to \infty} \cos(\lambda_n) - \frac{i\lambda_n}{\omega B} &= 0.
    \label{coslambda}
\end{align}
The second summand of \eqref{coslambda} clearly tends to zero. Now see that
\begin{align*}
    \cos(\lambda_n) &= \cos\left( \frac{(4n+1)\pi}{2} \right) \cosh\left( \frac{(4n+1)\pi}{2 \omega B} \right) - i \sin\left(\frac{(4n+1)\pi}{2} \right)\sinh\left( \frac{(4n+1)\pi}{2\omega B} \right) \nonumber \\
    &=i \sinh\left( \frac{(4n+1)\pi}{2 \omega B} \right) \to 0 
\end{align*}
as $\omega \to \infty$.

Next, consider the Taylor series of $\cos(z) - \frac{iz}{\omega B}$ centered at $z_n$. This gives us 
\begin{align}
    \cos(z) - \frac{iz}{\omega B} &= \left(\cos(z_n) - \frac{iz_n}{\omega B}\right) - \left(\sin(z_n) + \frac{iz_n}{\omega B}\right) (z - z_n) \nonumber \\
    &\-\ + \sum\limits_{k \geq 1}^\infty \left[ \frac{(-1)^{k}\cos(z_n)}{(2k)!}(z-z_n)^{2k} + \frac{(-1)^{k+1} \sin(z_n)}{(2k+1)!}(z-z_n)^{2k+1} \right]. \label{F1expansion}
\end{align}
As $\cos(z_n) = \frac{iz_n}{\omega B}$, we have that $\sin(z_n) = \sqrt{1 + \frac{z_n^2}{\omega^2B^2}}$. So as $\omega \to \infty$, $\cos(z_n) \to 0$ and $\sin(z_n) \to 1$. Substituting $z = \lambda_n$ and taking $\omega \to \infty$ in \eqref{F1expansion}, while applying \eqref{coslambda} to the lefthand side gives
\begin{align}
    0 &= (\lambda_n - z_n) \left[-1 + \frac{1}{3!}(z-z_n)^2 - \frac{1}{5!}(z-z_n)^4 + \dots \right] \nonumber \\
    &= (z_n- \lambda_n) \frac{\sin(\lambda_n - z_n)}{\lambda_n - z_n} \\
    &= z_n - \lambda_n.
\end{align}
We can now see that $\lambda_n$ is asymptotically equivalent to $z_n$.
\end{proof}
The asymptotic equivalence of the poles is what allows us to evaluate the limit of the residues as the residues of the limit. Notice that $\widetilde{F_-}$ converges uniformly to $\widetilde{F^\infty_-}(z) = -\frac{i (2B-1)}{8B^2z\cos(z)}$ on compact subsets of $\C - \Lambda^-(\infty)$, whilst $\widetilde{F_+}$ converges uniformly to $\widetilde{F^\infty_+}(z) = \frac{i (2B-1)}{8B^2z\cos(z)}$ on compact subsets of $\C - \Lambda^+(\infty)$. It is an elementary computation to confirm that
\begin{align}
    \text{Res}\left( \widetilde{F^\infty_-}(z), \frac{(4n-1)\pi}{2} \right) &= \frac{-i(2B-1)}{4B^2(4n-1)\pi}, n \in \Z - \{0\} \\
    \text{Res}\left( \widetilde{F^\infty_+}(z), \frac{(4n-1)\pi}{2} \right) &= \frac{-i(2B-1)}{4B^2(4n+1)\pi}, n \in \Z - \{0\} \\
    \text{Res}\left( \widetilde{F^\infty_+}(z), \pm \frac{\pi}{2} \right) &= \frac{-i(2B-1)}{4B^2\pi}.
\end{align}
Multiplying these residues by $2\pi i$ and summing them yields
\begin{align}
    \lim\limits_{\omega \to \infty} M(\omega, B) &= \frac{(2B-1)}{B^2} \sum\limits_{k=0}^\infty \frac{1}{2k+1}, \label{residuesum}
\end{align}
which will prove our result.

To conclude the proof, it suffices to describe the contours. For $R > 0$, let $\eta_R$ be the positively oriented square with vertices $-R, R, R + iR, -R + iR$. Label the four sides $\eta_1, \eta_2, \eta_3$, and $\eta_4$ counter clockwise, with $\eta_1$ being the side length from $-R$ to $R$. Notice that $\int_{\eta_2} F(z)dz + \int_{\eta_4} F(z)dz = 0$. Finally, we can see that
\begin{align*}
    \left| \int_{\eta_3} F(z)dz \right| &\leq 2R \left| \sup\limits_{z \in \eta_3} F(z) \right| \\
    &\leq \frac{2CR}{R^2} \to 0,
\end{align*}
as $R \to \infty$, establishing \eqref{residuesum} and completing the proof.
\end{proof}

\section{M\"obius-Plateau Stationary Symmetric Screws}

We now discuss the implications of Theorem \ref{MwB} to the more general M\"obius-Plateau energy. The classical Plateau problem seeks to find a surface minimizing surface area whose boundary is a given fixed $C^2$ curve $\gamma$. This minimal surface area is defined as the Plateau energy $E_P(\gamma)$. The existence of minimal surfaces with given boundary was proven independently by Douglas and Rado, and an exposition of their proof can be found in \cite{colding2011course}. Plateau problems with free boundary subject to energy constraints have been of recent theoretical and applied interest \cite{matsutani2010euler,giusteri2017solution}. In particular, the Euler-Plateau problem seeks to minimize the sum of a curve's Plateau energy and its elastic energy, given by the integral of the squared curvature of $\gamma$. Minimizing this energy has applications to the study of cellular membranes. Bernatzki and Ye proved the existence of minimizers of the Euler-Plateau energy, but the gradient descent of curve could intersect unless one assumes the initial curve has sufficiently low energy to preclude intersections \cite{bernatzki2001minimal}.

As the M\"obius energy is self-repulsive and possesses strong regularity properties \cite{freedman1994mobius,blatt20}, this leads us to believe that the M\"obius-Plateau energy, defined as the sum of the M\"obius and Plateau energies, could provide an alternative to the strong assumptions on the initial conditions. 

The gradient of the Plateau energy of $\gamma$ is given by $\textbf{n} = \textbf{T} \times \boldsymbol{\nu}$, where $\textbf{T}$ is the Frenet tangent vector and $\boldsymbol{\nu}$ is a choice of unit normal for the minimal surface bounded by $\gamma$. When there are two boundary components, we orient the curves oppositely so the Plateau gradient vector field pulls the surface apart, increasing the minimal surface area. Given fixed constants $\alpha, \beta \geq 0$, the general M\"obius-Plateau energy of the pair of curves $\gamma_1,\gamma_2$ is defined by
\begin{align}
    E(\gamma_1,\gamma_2) &= \alpha E_M(\gamma_1,\gamma_2) + \beta E_P(\gamma_1, \gamma_2).
\end{align}
So the variational equation of the M\"obius-Plateau energy is given by
\begin{align}
    \alpha G_{\gamma_1,\gamma_2} &= -\beta \textbf{n}.
    \label{mobiusplateauvar}
\end{align}

In joint work with Nair \cite{ln22}, we considered the M\"obius-Plateau energy of helicoidal strips, bounded by two helices on the same helicoid. Helicoidal strips are classified as either screws or ribbons, depending on whether or not they contain the axis. This distinction alters the variational equations and the criteria for the strips to be stationary. We proved a strong characterization for stationary ribbons: the coiling must be high, the width must be thin in comparison to the coiling, and the ribbon must remain close to the axis. 

The symmetric helix pairs we have considered are the boundary of a helicoidal screw. A general screw is determined by the radii $A$ and $B$, with frequency $\omega$, where we now assume $A < 0 < B$ so the axis is included. For a screw to be stationary, the repulsive forces on the helices seeking to decrease the M\"obius energy is must be cancelled by the attractive force seeking to contract the helices and decrease the surface area. 

\begin{figure}[h]
    \centering
    \includegraphics[width=6in]{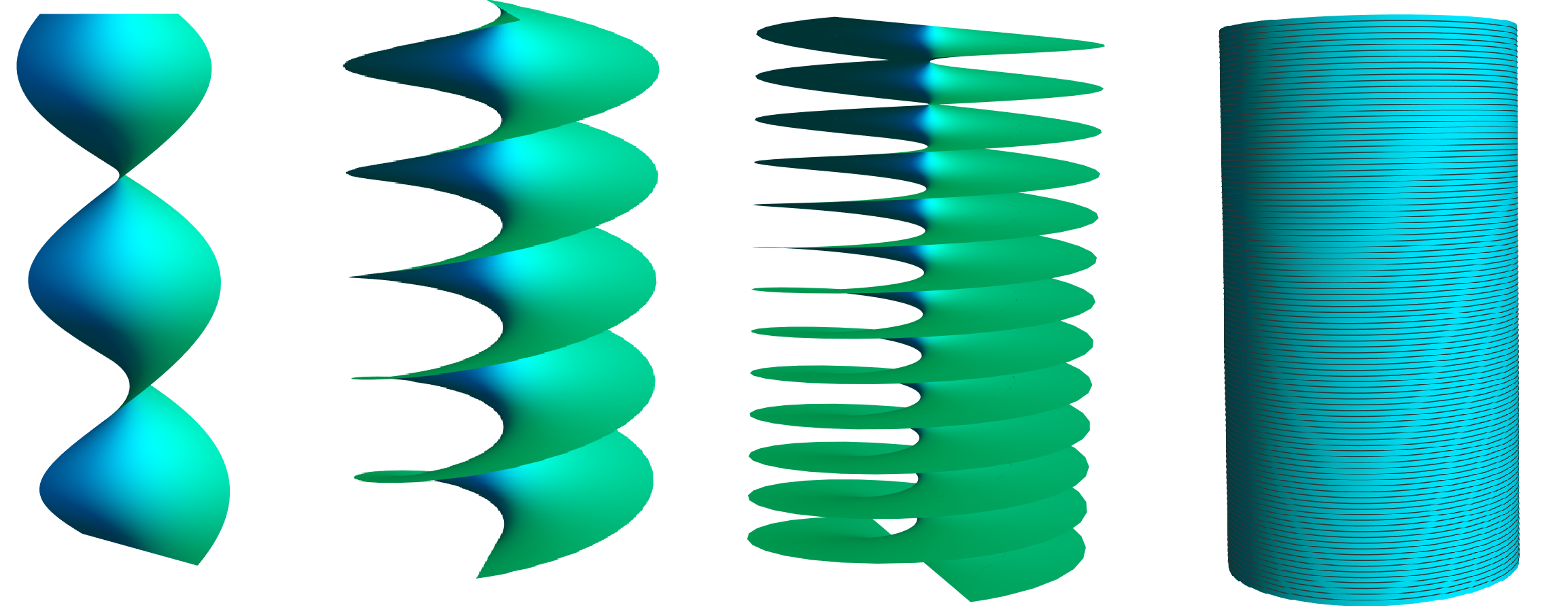}
    \caption{As $\omega \to \infty$, stationary symmetric screws approach a radius of $\frac{1}{2}$ from below, tending towards a limiting helicoid with near-flat curvature away from the axis, and infinite Gaussian curvature at the axis, all whilst maintaining zero mean curvature due to helicoids being minimal. These are not the actual calculated screws. This image is for illustrative purposes.}
    \label{symscrewlimits}
\end{figure}

Applying the variational equation \eqref{mobiusplateauvar} to helicoidal screws, yields the system
\begin{align}
        4 \alpha \int_{-\infty}^\infty \frac{(A + B)(\cos(\omega v) - 1)}{\left(A^2 - 2AB \cos(\omega v) + B^2 + v^2 \right)^2}dv &= \beta \left( \frac{1}{\sqrt{\omega^2 B^2 + 1}} - \frac{1}{\sqrt{\omega^2 A^2 + 1}}\right). \nonumber \\
         4\alpha \int_{-\infty}^\infty \frac{A(A+1) + B(B-1) + (A - B -2AB) \cos(\omega v) + v^2}{\left(A^2 - 2AB \cos(\omega v) + B^2 + v^2 \right)^2}dv &= -\beta \left( \frac{1}{\sqrt{\omega^2 B^2 + 1}} + \frac{1}{\sqrt{\omega^2 A^2 + 1}}\right).
         \label{mpscrews}
\end{align}

The Plateau gradient is a unit vector field, but the terms on the righthand sides come from the arclength parameters of the integrals. When we were considering only the M\"obius energy, we could divide by these constants to simplify the equation, but we do not have that luxury here. Again, we refer the interested reader to \cite{ln22} for a derivation of this system of equations.

Like before, we can see that setting $A = -B$ sets both sides of the first equation in \eqref{mpscrews} to zero, suggesting there could be stationary symmetric screws. Making this substitution in the second equation yields
\begin{align}
     2\alpha \bigintsss_{-\infty}^\infty \frac{4B(B-1)\cos^2\left(\frac{\omega v}{2}\right) + v^2}{\left(\left(2B\cos\left(\frac{\omega v}{2}\right)\right)^2  + v^2 \right)^2}dv &= \frac{-\beta}{\sqrt{\omega^2B^2 + 1}}.
     \label{symscrewdiff}
\end{align}
The quantity on the right hand side of \eqref{symscrewdiff} is negative, but tends to zero as $\omega \to \infty$. By applying Theorem \ref{MwB}, we can get the following characterization of the asymptotics of M\"obius-Plateau stationary screws, as depicted in Figure \ref{symscrewlimits}.

\begin{theorem}
If $\omega_i$ is a sequence of frequencies tending to infinity, then there exists a sequence $B_i$ tending to $\frac{1}{2}$ from below such that for all but finitely many $i$, the symmetric screw with radius $B_i$ and frequency $\omega_i$ is stationary under the M\"obius-Plateau energy.
\label{symscrewtheorem}
\end{theorem}

\section*{Acknowledgements}
The author would like to thank Xin Zhou for his numerous helpful discussions with this work, and in particular making the initial suggestion to consider the M\"obius-Plateau energy. He would also like to thank Steven Strogatz and Gokul Nair for their helpful comments.
\\
\\
\noindent This work is partially funded by an NSF RTG grant entitled Dynamics, Probability and PDEs in Pure and Applied Mathematics, DMS-1645643.

\bibliographystyle{unsrt}
\bibliography{ref}

\Addresses
\end{document}